\documentclass[notitlepage]{amsart}

\usepackage{amsfonts}
\usepackage[latin1]{inputenc}
\usepackage[all]{xy} 
\usepackage{latexsym,amssymb,amscd}
\usepackage{amsmath}
\usepackage{mathrsfs}
\usepackage{multicol}

\usepackage{pb-diagram, pb-xy}

\usepackage{graphicx}
\usepackage{fancyhdr}
\usepackage{color}
\usepackage[T1]{fontenc}

\usepackage{tikz}
\usetikzlibrary{arrows}
 \usetikzlibrary{snakes}
 \usetikzlibrary{shapes}

\newtheorem{pro}{Proposition}[section]
\newtheorem{teo}[pro]{Theorem}
\newtheorem{lem}[pro]{Lemma}
\newtheorem{defi}[pro]{Definition}
\newtheorem{cor}[pro]{Corollary}
\newtheorem{rk}[pro]{Remark}

\usepackage{latexsym,amssymb,amscd}
\usepackage{amsmath}
\begin{document}
\title[Cardinality of the Dickson permutation's group of polynomials on $\mathbb{Z}_n$ ]{Cardinality of the Dickson permutation's group of polynomials on $\mathbb{Z}_n$}
\author{L.Pe\~na, M. Ort\'iz}
\thanks{
The authors thank CONACyT}
\maketitle

\date{}
\begin{abstract}
Let $G_n$ be the group of permutations on $\mathbb{Z}_n$ that is induced by a Dickson polynomial, where $n$ is a positive integer. In this work, by solving special types of systems of linear congruences, we obtain $G_n$. In addition, we give an a algorithm that gets $|G_n|$.
\end{abstract}


\section*{Introduction}
\label{intro}

Dickson's polynomials are a specific case of Waring's formula for a quadratic equation, which was published by Edward Waring in 1762 and is a formula for the sum of the $k$-th power of the roots of a polynomial equation $$x^n + c_1 x^{n-1} + \cdots + c_n =0,$$ which is calculated as

\begin{equation*}
s_k = k \sum (-1)^{r_1 + r_2 + \cdots + r_n} \frac{(r_1 + r_2 + \cdots + r_n - 1)!}
{r_1 !\cdots r_n !} {c}^{r_1}_{1} \cdots {c}^{r_n}_{n},
\label{eq1}
\end{equation*}
where the sum extends over the whole set of integers $r_1 , r_2 , \cdots , r_n \geq 0$ such that $r_1 + 2r_2 + \cdots + nr_n = k$ (cf. \cite{dickson}).

As part of his Ph.D. thesis at the University of Chicago in 1896, L.E. Dickson began the study of a class of polynomials of the form $$x^n + n \sum_{i=1}^{(n-1)/2} \dfrac{(n-i-1)\ldots(n-2i+1)}{i!}a^ix^{n-2i},$$ over finite fields where $n$ is odd. I. Schur named these polynomials in honor of Dickson and observed that the polynomials are related to the Chebyshev polynomials. Shur's paper (cf. \cite{schur}) from 1923 also gave rise to the conjecture that the only polynomials with integer coefficients that induce permutations of the integers mod $p$ for infinitely many primes $p$ are compositions of linear polynomials, power polynomials $x^n$, and Dickson polynomials. W.B. N\"obauer in \cite{nobauer} shows many properties of Dickson polynomials over finite fields and over $\mathbb{Z}_n$.

In the last years Dickson permutation polynomials have been used in cryptography, in the process of designing systems for the transmission of information in a secure way and as a key exchange protocol (cf. \cite{leticia}). The Dickson cryptosystem is more general than the RSA cipher, since for the Dickson scheme the modulus $n$ need not be squarefree, but can be an arbitrary positive integer with at least two prime factors. In addition, Dickson polynomials have been used for primality tests in number theory (cf. \cite{dicksonpolynomials}).

The Dickson polynomial of first kind and degree $k$ with parameter $a$ on same ring $R$ is defined as 
\begin{equation*}
D_k(x,a)= \sum_{j=0}^{\lfloor\frac{k}{2}\rfloor}\frac{k}{k-j}\binom{k-j}{j} (-a)^{j}x^{k-2j},
\label{eq21}
\end{equation*}
where $k \in \mathbb{N}$ and $D_0(x,a)=2$.

Let $G_n$ be the set of all permutations on $\mathbb{Z}_n$ that are induced by a Dickson polynomial with parameter $a=1$. It is know that $G_n$ is an commutative group (cf. \cite{nobauer}), and it is of great importance to determine how many Dickson permutation polynomials induce a permutation on $\mathbb{Z}_n$. Furthermore, we are also interested in the fact that these permutations induced by the Dickson polynomials are different from the identity permutation.

\bigskip
In this work, we calculate the number of different permutations in $G_n$. In order to do this we first determine $G_{p^e}$, where $p\geq 3$ is prime and $e$ is a positive integer; we then define an epimorphism $\psi: \mathbb{Z}^*_{p^{e-1}\left(\frac{p^2 -1}{2}\right)}\rightarrow G_{p^e}$ of Dickson polynomials that induce permutations on $\mathbb{Z}_{p^e}$. We then calculate $\mathrm{Ker}\, \psi$, and so we obtain $\vert G_{p^e} \vert = \left\vert\frac{\mathbb{Z}^*_{p^{e-1}\left(\frac{p^2 -1}{2}\right)}}{\mathrm{Ker}\,\psi}\right\vert$.

For $G_{2^e}$, we define epimorphism $\psi': \mathbb{Z}^*_{3 \cdot 2^{e-1}}\rightarrow G_{2^e}$ if $e<3$ and $\psi'': \mathbb{Z}^*_{3 \cdot 2^{e-2}}\rightarrow G_{2^e}$ if $e\geq 3$. Thus, $\vert G_{2^e} \vert = \left\vert \frac{\mathbb{Z}^*_{3 \cdot 2^{e-1}}}{\mathrm{Ker}\,\psi} \right\vert$ and $\vert G_{2^e} \vert = \left\vert \frac{\mathbb{Z}^*_{3 \cdot 2^{e-2}}}{\mathrm{Ker}\,\psi} \right\vert$.

Some of the results of this work are obtained following the ideas in \cite{nobauer,leticia}. In \cite{nobauer}, Laush, Muller and N\"obauer get $K_{p^e}:=\mathrm{Ker}\, \psi$, $K_{2^e}:=\mathrm{Ker}\, \psi'$ and $K_{2^e}:=\mathrm{Ker}\, \psi''$; consequently, $ |K_{2^e}|=2$, and if $p=3$ or $e>1$, then $|K_{p^e}|=2$ and if $p\geq 5$ and $e=1$ then $|K_{p^e}|=4$. On the other hand in \cite[Ch.3]{leticia}, we have the following values for the cardinality of $G_{p^e}$:
\begin{equation*}
|G_{p^e}|= \left\lbrace 
\begin{array}{ll}
1, & \text{if } p=2 \text{ and } e<3, \\
2^{e-3}, & \text{if } p=2 \text{ and } e\geq 3, \\
1, & \text{if } p=3 \text{ and } e=1, \\
2\cdot 3^{e-2}, & \text{if } p=3 \text{ and } e>1, \\
\frac{1}{4}\cdot\varphi\left( \frac{p^2-1}{2}\right) , & \text{if } p\geq5 \text{ and } e=1, \\
\frac{p^{e-2}(p-1)}{2}\cdot\varphi\left( \frac{p^2-1}{2}\right) , & \text{if } p\geq5 \text{ and } e>1.
\end{array}\right.
\end{equation*}

In \cite[Ch.4]{dicksonpolynomials}, they consider $n=p_1^{e_1}\cdots p_r^{e_r}$, where $p_i$ is a prime number and $e_i\geq 1$ an integer for $i=1,\ldots,r$. Furthermore, they define $v(n)=\mathrm{\mathrm{lcm}\,}(p_1^{e_1-1}(p_1^2 -1), \ldots,p_r^{e_r-1}(p_r^2 -1))$. In addition, they prove that if $a \in \mathbb{Z}_n^{*}$, then the Dickson polynomial $D_k(x,a)$ is a permutation on $\mathbb{Z}_n$ if and only if $\gcd(k,v(n))=1$.

For this work, then let $n=2^e\cdot p_1^{e_1}\cdots p_r^{e_r}$ where $p_i$ is a prime odd, and $e\geq0$ and $e\geq1$ are integers. We define $l_i:=p_i^{e_i-1}\left( \frac{p_i^2-1}{2} \right)$ for $i=1,\ldots,r$ and $$l_0 := \left\lbrace 
\begin{array}{ll}
3\cdot 2^{e-1}, & \text{if } 1\leq e <3, \\
3 \cdot 2^{e-2}, & \text{if } e\geq3.
\end{array}\right.$$

In this way analogously to $v(n)$ defined above, we define $w(n)$ by
\begin{equation*}
w(n) = \left\lbrace 
\begin{array}{ll}
\mathrm{\mathrm{lcm}\,}(l_1,\ldots,l_r), & \text{if } e=0,\\
\mathrm{\mathrm{lcm}\,}(l_0,l_1,\ldots,l_r), & \text{if } e\geq1.
\end{array}\right.
\label{eq30}
\end{equation*}

In this work, we prove that if $a \in \mathbb{Z}_n^{*}$, then the Dickson polynomial $D_k(x,a)$ is a permutation on $\mathbb{Z}_n$ if and only if $\gcd(k,w(n))=1$. 

By using similar arguments to calculate $G_{p^e}$, we obtain $G_n$. For this, we define an epimorphism $\xi: \mathbb{Z}_{w(n)}^{*} \rightarrow G_n$ of Dickson polynomials that induce permutations on $\mathbb{Z}_n$. Therefore, $|G_n|=\vert \mathbb{Z}_{w(n)}^{*}/ K_n \vert$, where $K_n := \mathrm{Ker}\, \xi$.

In order to determine $|K_n|$, we consider the congruence systems:
\begin{equation}
x \equiv a_1 \pmod {l_1}, \ldots , x \equiv a_r \pmod {l_r},
\label{eq22}
\end{equation}
and 
\begin{equation}
x \equiv a_0 \pmod {l_0}, x \equiv a_1 \pmod {l_1}, \ldots, x \equiv a_r \pmod {l_r},
\label{eq23}
\end{equation}
where $a_0 \in \mathbb{Z}_{l_0}^{*}$ and $a_i \in \mathbb{Z}_{l_i}^{*}$ for all $i=1,\ldots,r$. We then consider the sets:
$$A_1=\lbrace(a_1,\ldots ,a_r) \in K_{p_1^{e_1}}\times \cdots \times K_{p_r^{e_r}} : \text{(\ref{eq22}) has a solution}\rbrace, \text{ and}$$ $$A_0=\lbrace(a_0,a_1,\ldots ,a_r) \in K_{2^e}\times K_{p_1^{e_1}}\times \cdots \times K_{p_r^{e_r}} : \text{(\ref{eq23}) has a solution}\rbrace.$$

If $e=0$, we define 
\begin{equation}
\begin{array}{lcll}
\rho_1: & A_1 & \rightarrow & K_n \\
 & (a_1,\ldots,a_r) & \mapsto & k, 
\end{array} 
\label{eq24}
\end{equation}
where $k$ is solution of (\ref{eq22}), and if $e \neq 0$, we define
\begin{equation}
\begin{array}{lcll}
\rho_0: & A_0 & \rightarrow & K_n \\
 & (a_0,a_1,\ldots,a_r) & \mapsto & k, 
\end{array} 
\label{eq25}
\end{equation}
where $k$ is solution of (\ref{eq23}).

In this work, we prove that the functions $\rho_1$ and $\rho_0$ are bijective, as is established in the following theorem. 

\vspace{0.3cm}
\textbf{Theorem.} \textit{The functions $\rho_1$ and $\rho_0$ defined in (\ref{eq24}) and (\ref{eq25}), respectively, are bijective functions.}\vspace{0.3cm}

Let 
\begin{equation}
A:= \left\lbrace \begin{array}{ll}
A_0, & \text{if } e\neq0, \\
A_1, & \text{if } e=0.
\end{array}\right.
\label{eq31}
\end{equation} 

As a consequence, in a natural way, there is a bijection $\rho: A \rightarrow K_n$, which gives to $A$ a group structure. Moreover, $A\simeq K_n$ and $|G_n|=\frac{|\mathbb{Z}_{w(n)}^{*}|}{|A|}$.

Finally, we give algorithms to determine $A$, solving all possible congruence systems of the form (\ref{eq22}) or (\ref{eq23}).

\bigskip
This work is divided in six sections. In Section \ref{sec:2}, we recall some basic results that we use through this work. In Section \ref{sec:3}, we define $w(n)$ as above, and we obtain results to prove the following: if $a \in \mathbb{Z}_n^{*}$, then the Dickson polynomial $D_k(x,a)$ is a permutation polynomial on $\mathbb{Z}_n$ if and only if $\gcd(k,w(n))=1$ (Theorem \ref{theorem7}). This result allows us to determine when a Dickson polynomial is a permutation on $\mathbb{Z}_n$, where $n$ is an arbitrary positive integer. In Section \ref{sec:4}, we discuss some results given in \cite{nobauer} to obtain $G_{2^e}$ and $G_{p^e}$. In Section \ref{sec:5}, we obtain one of the main results of this work. In order to determine $G_n$ following the ideas of Section \ref{sec:4}, an epimorphism $\xi: \mathbb{Z}_{w(n)}^{*} \rightarrow G_n$ is defined. Next we obtain $K_n: = \mathrm{Ker}\, \xi$. To compute $|K_n|$, we give a bijection $\rho: A \rightarrow K_n$ as we establish in the following theorem, where $A$ is defined as in (\ref{eq31}).

\vspace{0.3cm}
\textbf{Theorem.} \textit{Let $\rho: A \rightarrow K_n$ be the function given by  
$$\rho = \left\{
\begin{array}{ll}
\rho_0  & \text{, if } e\neq 0,\\
\vspace{0.1cm} & \\
\rho_1  & \text{, if } e=0.
\end{array} \right.
$$
Thus, $\rho$ is a bijective function.}\vspace{0.3cm}

In Section \ref{sec:6}, we describe the algorithms we developed to be able to solve the congruence systems (\ref{eq22}) or (\ref{eq23}) and count all their solutions. In this way, we were able to determine $K_n$.

\section{Preliminaries}
\label{sec:2}

In this section, we recall some basic results regarding regular polynomials. We see that $f(x)=x^2-ux+1$ is a regular polynomial, and the sum of its $k$-th power of the roots determines a Dickson polynomial $D_k(x,1)$ that we use in this work. Since the roots of  $f(x)=x^2-ux+1$  lie in $\mathbb{Z}_n$ or some extension of the $\mathbb{Z}_n$ ring,  in the proofs of Section \ref{sec:4} we use diagram chasing arguments. We give a special commutative diagram, which will be a useful tool. In addition, some basic properties of the Dickson polynomials are mentioned, as well as how it is possible to define a composition operation between Dickson polynomials. In addition, in this section we also mention a theorem that determines when a Dickson polynomial is a permutation on $\mathbb{Z}_n$ (cf. \cite{nobauer}).

\subsection{Regular polynomials}

Let $p$ be a prime number and $e$ a positive integer. Since in Section \ref{sec:4} we consider the case for $n=p^e$, where $p$ is a prime number, the following list of results will be useful, which are proved in \cite{atiyah} and \cite{biniflam}.

\begin{lem}
$\mathbb{Z}_{p^e}$ is a local ring: $\mathbb{Z}_{p^e}$ has a unique maximal ideal.
\label{lemma1}
\end{lem}

\begin{lem}
Let $\mathbb{Z}_{p^e}/p\mathbb{Z}_{p^e}$ be the residue field of $\mathbb{Z}_{p^e}$. Consider
\begin{equation*}
\begin{array}{ll}
\varphi: & \mathbb{Z}_{p^e} \rightarrow \mathbb{Z}_p \\
      & \bar{a} \mapsto \bar{\bar{a}},
\end{array}
\label{eq2}
\end{equation*} 
where $\bar{a}=a \mod p^e$ and $\bar{\bar{a}}=a \mod p$ for $a \in \mathbb{Z}$. As a result, $\mathrm{Ker}\, \varphi \simeq \mathbb{Z}_{p^e}$, $\mathrm{Im}\, \varphi \simeq \mathbb{Z}_p$ and thus $\mathbb{Z}_p \simeq \mathbb{Z}_{p^e}/p\mathbb{Z}_{p^e}$.
\label{lemma35}
\end{lem}

Let $R$ be a finite, commutative local ring, with a unique maximal ideal $M$ and a residue field $K=R/M$. The canonical projection $\pi: R \rightarrow K$, $a \mapsto a+M$ extends to a morphism of polynomial rings:

\begin{equation}
\begin{array}{rl}
\mu: R[x]& \rightarrow K[x] \\
     f(x)= \sum_{i=1}^{n} a_i x^i & \mapsto \sum_{i=1}^{n} \pi(a_i) x^i. 
\end{array}
\label{eq3}
\end{equation}

If $R$ is a commutative ring, an ideal $I$ of $R$ is said to be \textbf{primary} if $I\neq R$ and, whenever $xy \in I$ and $x\notin  I$, $y^n \in I$ for some positive integer $n$.

\begin{defi}
Let $f$ and $g$ be elements of $R[x]$.
\begin{itemize}
\item [$\mathrm{1}.$ ] $f$ is regular if it is not a zero divisor.
\item [$\mathrm{2}.$ ] $f$ is primary if $\langle f\rangle$ is a primary ideal.
\item [$\mathrm{3}.$ ] $f$ and $g$ are relatively prime if $R[x]= \langle f\rangle + \langle g\rangle $.
\end{itemize}
\label{def1}
\end{defi}

The following results appear in \cite{biniflam}.

\begin{pro}
Let $f(x)= \sum_{i=1}^{n} a_i x^i \in R[x]$. The following conditions are equivalent:

\begin{tabular}{rl}
 $\mathrm{i)}$ & $f$ is a unit. \\
 $\mathrm{ii)}$ & $\mu(f)$ is a unit in $K[x]$. \\
 $\mathrm{iii)}$ & $a_0$ is a unit in $R$ and $a_1,\dots,a_n$ are nilpotent.
 \end{tabular}  
\label{Proposition2}
\end{pro}

\begin{teo}
Let $f$ be a regular polynomial in $R[x]$. If $\mu(f)$ is irreducible in $K[x]$, then $f$ is irreducible in $R[x]$. 
\label{theorem3}
\end{teo}

Let $f(x)=x^2-ux+1 \in \mathbb{Z}_{p^e}$, where $\mathbb{Z}_{p^e}$ is a finite, commutative and local ring with a unique maximal ideal $p\mathbb{Z}_{p^e}$ by Lemma \ref{lemma1}.

\begin{pro}
$f(x)=x^2-ux+1 \in \mathbb{Z}_{p^e}[x]$ is not a zero divisor, and therefore $f(x)$ is a regular polynomial. 
\label{observation4}
\end{pro}

\begin{proof}
Assume that there exists $h(x)=\sum_{i=0}^{t} a_ix^i \in \mathbb{Z}_{p^e}[x]$, such that the product of polynomials $f(x)h(x)$ is equal to zero. Thus, 

\begin{equation*}
\begin{array}{rl}
f(x)h(x) & =x^2\sum_{i=0}^{t} a_ix^i -ux\sum_{i=0}^{t} a_ix^i+\sum_{i=0}^{t} a_ix^i \\
 & = \sum_{i=2}^{t+2} a_{i-2}x^i-\sum_{i=1}^{t+1} a_{i-1}ux^i+\sum_{i=0}^{t} a_ix^i.
\end{array}
\end{equation*}

Therefore,
$$f(x)h(x)=a_tx^{t+2}+(a_{t-1}-a_tu)x^{t+1}+\sum_{i=2}^{t} (a_{i-2}-a_{i-1}u+a_i)x^i + a_1x+a_0=0.$$

As a consequence, $a_0=a_1=a_t=a_{t-1}-a_tu=0$ and $a_{i-2}-a_{i-1}u+a_i=0$ for $2\leq i \leq t$, therefore $a_i=0$ for all $i=0,\ldots,t$. Thus $h(x)=0$, so $f(x)$ is not a zero divisor, and thus $f(x)$ is a regular polynomial by Definition \ref{def1}.
\end{proof}

By Theorem \ref{theorem3}, if $\mu(f)=x^2 - \bar{\bar{u}}x+1$ is irreducible in $K[x]\simeq \mathbb{Z}_p[x] $, then $f(x)=x^2-ux+1$ is irreducible in $\mathbb{Z}_{p^e}[x]$.

Since a Dickson polynomial $D_k(x,a)$ is the sum of the $k$-th power of the roots of the quadratic equation $f(x)=x^2-ux+1$, it is necessary to know the roots of $f(x)$ to be able to see the polynomials of Dickson as in (\ref{eq4}). The roots of $f(x)=x^2-ux+1 \in \mathbb{Z}_{p^e}[x]$ lie in $\mathbb{Z}_{p^e}$ or in $\mathbb{Z}_{p^e}[x]/\langle f(x)\rangle$. In the proofs of Section \ref{sec:4}, diagram chasing arguments are used, and for this reason we need the following commutative diagram with exact columns, where $\mu$ and $\bar{\mu}$ are induced by $\varphi$:

\begin{center} 
\hspace{1cm}
\xymatrix{
\mathbb{Z}_{p^e} \ar@{^{(}->}[d] \ar[r]^{\varphi} & \mathbb{Z}_p \ar@{^{(}->}[d] \\
\mathbb{Z}_{p^e}[x] \ar@{->>}[d] \ar[r]^{\mu} & \mathbb{Z}_p [x] \ar@{->>}[d] \\
\mathbb{Z}_{p^e}[x]/\langle f \rangle \ar[r]^{\bar{\mu}} & \mathbb{Z}_p [x]/ \langle \mu (f)\rangle .}
\end{center}

\subsection{Dickson polynomials}

We present some properties of the Dickson polynomials that are satisfied in any ring (cf. \cite{dicksonpolynomials}, \cite{leticia}), which will be of importance for the development of this work.

Assume that $y$ be a solution of the equation $x^2-ux+1 = 0$ in the ring $R$ or in some extension of $R$. Then $y\neq 0$ and $1=y(u-y)$: $y$ is invertible in $R$ or in the extension of $R$. Therefore, by multiplying $y^{-2}$ by $y^2-uy + 1 = 0$, we have $$1-u\left( \frac{1}{y}\right) + \left( \frac{1}{y}\right)^2=0.$$ 

Thus, $\frac{1}{y}$ is also a solution of the quadratic equation $x^2-ux + 1=0$, and 
$$x^2-ux+1=\left( x-y \right) \left( x-\frac{1}{y} \right) ,$$ where $u=y+\frac{1}{y}$; furthermore, by Waring's formula,
\begin{eqnarray*}
y^k + \left( \frac{1}{y} \right)^k & = & \sum^{\left\lfloor \frac{k}{2} \right\rfloor}_{j=0} \frac{k}{k-j} \binom {k-j} {j} \left( - y \left(\frac{1}{y}\right) \right)^j \left( y + \frac{1}{y}\right)^{k-2j} \\
& = & \sum^{\left\lfloor \frac{k}{2} \right\rfloor}_{j=0} \frac{k}{k-j} \binom {k-j} {j} (-1)^j \left( y + \frac{1}{y}\right)^{k-2j}. 
\end{eqnarray*}

Therefore,
\begin{equation}
D_k\left( y+\frac{1}{y},1 \right)= y^k + \left( \frac{1}{y} \right)^k.
\label{eq4}
\end{equation}

Dickson polynomials can be calculated recursively with the following lemma shown in \cite{dicksonpolynomials}.

\begin{lem}
Let $D_k(x,a)$ be a Dickson polynomial, then for $k\geq 2$ the following recurrence relation is satisfied: $$D_k(x,a) = xD_{k-1}(x,a)- a D_{k-2}(x,a),$$ with initial values $D_0(x,a)=2$ and $D_1(x,a) = x$.
\label{lemma4}
\end{lem}

If it is desired to evaluate the Dickson polynomial in some ring $\mathbb{Z}_{n}$, doing so recursively would be costly; \cite{leticia} shows a Dickson polynomial fast evaluation algorithm.

The following lemma allow us to define a composition operation between Dickson polynomials; see \cite{dicksonpolynomials}.

\begin{lem}
The Dickson polynomials satisfy the following properties:

\begin{tabular}{rl}
 $1.$ & $D_n(x,0)=x^n$. \\
 $2.$ & $D_{mn}(x,a)=D_m(D_n(x,a),a^n)$. 
 \end{tabular} 
 
\label{lemma5}
\end{lem}

\bigskip

Let 
\begin{equation*}
n=p_{1}^{e_1}\cdots p_{r}^{e_r},
\label{eq14}
\end{equation*}
where $p_i$ different prime numbers for $1\leq i \leq r$ and $e_i \in \mathbb{N}$; see \cite[Ch.4]{dicksonpolynomials}. In addition,

$$v(n):=\mathrm{lcm}\,\left(p_1^{e_1-1}(p_1^2-1),\ldots,p_r^{e_r-1}(p_r^2-1) \right).$$ 

The following theorem allow us to determine when a Dickson polynomial is a permutation polynomial on $\mathbb{Z}_n$; see \cite[Ch.4]{dicksonpolynomials}.

\begin{teo}
The Dickson polynomial of degree $k$, $D_k(x,a)$ with $a \in \mathbb{Z}_n^{*}$ is a permutation polynomial on $\mathbb{Z}_n$ if and only if $\gcd(k,v(n))=1$.
\label{theorem6}
\end{teo}

\section{Dickson permutation polynomials\textbf{$\mod n$}}
\label{sec:3}

In this section, we define $w(n)$, and we prove results that help to prove the main theorem of this section that establishes the following: if $a \in \mathbb{Z}_n^{*}$, then $D_k(x,a)$ is a permutation polynomial on $\mathbb{Z}_n$ if and only if $\gcd(k,w(n))=1$. This result allows us to determine when a Dickson polynomial is a permutation on $\mathbb{Z}_n$, where $n$ is an arbitrary positive integer.

In this work, we consider
\begin{equation}
n=2^e p_{1}^{e_1}\cdots p_{r}^{e_r},
\label{eq40}
\end{equation} with $e$ an integer greater than or equal to $0$, $p_i$ different odd primes and $e_i \in \mathbb{N}$ for $i=1,\ldots,r$. We also define $l_i:=p_i^{e_i-1}\left( \frac{p_i^2-1}{2} \right)$ for $i=1,\ldots,r$ and 
\begin{equation}
l_0 := \left\lbrace 
\begin{array}{ll}
3\cdot 2^{e-1}, & \text{if } 1\leq e <3, \\
3 \cdot 2^{e-2}, & \text{if } e\geq3.
\end{array}\right.
\label{eq26}
\end{equation}

In this way, we define $w(n)$ as
\begin{equation}
w(n) := \left\lbrace 
\begin{array}{ll}
\mathrm{\mathrm{lcm}\,}(l_1,\ldots,l_r), & \text{if } e=0,\\
\mathrm{\mathrm{lcm}\,}(l_0,l_1,\ldots,l_r), & \text{if } e\geq1.
\end{array}\right.
\label{eq11}
\end{equation}

The following theorem is analogous to Theorem \ref{theorem6} in \cite{dicksonpolynomials}, because it helps us to determine when a Dickson polynomial is a permutation polynomial$\mod n$.

\begin{teo}
The Dickson polynomial of degree $k$, $D_k(x,a)$ with $a \in \mathbb{Z}_n^{*}$ is a permutation polynomial on $\mathbb{Z}_n$ if and only if $\gcd(k,w(n))=1$.
\label{theorem7}
\end{teo}

To prove the last theorem, some preliminary results are required, which are shown below.

\begin{pro}
Let $e,k \in \mathbb{N}$, $e \geq 1$ and $p$ be an odd prime. Then $$\gcd\left(k,p^{e-1}\left( p^2 -1\right)\right)=1 \text{ if and only if } \gcd\left(k,p^{e-1}\left( \frac{p^2 -1}{2}\right)\right)=1.$$
\label{Proposition26}
\end{pro}

\begin{proof}
First let us suppose that $\gcd\left(k,p^{e-1}\left( p^2 -1\right)\right)=1$, then there exist $a,b \in \mathbb{Z}$ such that $ak + bp^{e-1}\left( p^2 -1\right)=1$. Since $p$ is odd, $p^2-1 \equiv 0 \pmod 4$, $\frac{p^2 -1}{2} \in \mathbb{Z}$ and $ak + 2bp^{e-1}\left( \frac{p^2 -1}{2}\right)=1$. Thus, $\gcd\left(k,p^{e-1}\left( \frac{p^2 -1}{2}\right)\right)=1$.

Conversely, suppose $d:=\gcd\left(k,p^{e-1}\left( p^2 -1\right)\right)>1$. In addition, we have $\gcd\left(k,p^{e-1}\left( \frac{p^2 -1}{2}\right)\right)=1$. Thus, there exist $a,b \in \mathbb{Z}$ such that 
\begin{equation}
ak + bp^{e-1}\left( \frac{p^2 -1}{2} \right)=1.
\label{eq28}
\end{equation}
Since $\frac{p^2 -1}{2}$ is even, $a$ and $k$ must be odd. Thus, (\ref{eq28}) implies $$2ak+bp^{e-1}\left( p^2 -1\right)=2,$$ so $d|2$. Since $d>1$, then $d=2$ and $2|k$ which is not possible, therefore $d=1$. 
\end{proof}

\begin{pro}
Let $e,k \in \mathbb{N}$, with $e \geq 3$. Then $$\gcd\left(k,3\cdot 2^{e-1}\right)=1 \text{ if and only if } \gcd\left(k,3\cdot 2^{e-2}\right)=1.$$
\label{Proposition27}
\end{pro}

\begin{proof}
$\gcd\left(k,3\cdot 2^{e-1}\right)=1$ if and only if $k$ is odd and is not a multiple of $3$ if and only if $\gcd\left(k,3\cdot 2^{e-2}\right)=1$.
\end{proof}

The following definition and lemma are useful in the rest of section, and they appear in \cite{dicksonpolynomials}.

\begin{defi}
Let $r(x)=\frac{g(x)}{h(x)}$ be a quotient of relatively prime polynomials over $\mathbb{Z}$. Then $r(x)$ is called a \textbf{permutation function$\mod m$} if $h(i) \mod m$ is a prime residue class$\mod m$ for every integer $i$ and the associated function $r(i) \equiv h(i)^{-1}g(i) \pmod m$ $i=1,2,\ldots,m$, is a permutation of the residue classes$\mod m$. 
\label{definition11}
\end{defi}

We note that $g(x)$ is a permutation polynomial$\mod m$ if and only if $\frac{g(x)}{1}$ is a permutation function$\mod m$.

\begin{lem}
If $m=ab$, where $\gcd(a,b)=1$, then $g(x)$ is a permutation polynomial$\mod m$ if and only if $g(x)$ is a permutation polynomial$\mod a$ and$\mod b$.
\label{lemma28}
\end{lem}

With the help of Lemma \ref{lemma28}, the proof Theorem \ref{theorem7} is the following.

\begin{proof}[Proof of Theorem \ref{theorem7}]

We have two cases for $n$, $n$ is odd or even. First, assume $n$ is odd and $n=p_1^{e_1}\cdots p_r^{e_r}$ is a factorization into a product of powers of primes. The proof follows by induction on $r$.

If $r=1$, then $n=p_1^{e_1}$ and $\gcd(a,p_1^{e_1})=1$. As a consequence of Proposition \ref{Proposition26} and Theorem \ref{theorem6}, $D_k(x,a)$ is a permutation polynomial$\mod p_1^{e_1}$ if and only if $$\gcd\left(k,p^{e-1}\left( \frac{p^2 -1}{2}\right)\right)=1.$$

If $r=2$, then $n=p_1^{e_1}p_2^{e_2}$ with $\gcd(p_1^{e_1},p_2^{e_2})=1$. Since $\gcd(a,n)=1$ then $\gcd(a,p_1^{e_1})=1$ and $\gcd(a,p_2^{e_2})=1$. By Lemma \ref{lemma28}, $D_k(x,a)$ is a permutation polynomial$\mod n$ if and only if $D_k(x,a)$ is a permutation polynomial$\mod p_1^{e_1}$ and$\mod p_2^{e_2}$. For the previous case, this last happens if and only if $\gcd\left(k,p_1^{e_1-1}\left( \frac{p_1^2-1}{2}\right)\right)=1$ and $\gcd\left(k,p_2^{e_2-1}\left( \frac{p_2^2-1}{2}\right)\right)=1$, which is equivalent to $\gcd\left(k,p_1^{e_1-1}\left( \frac{p_1^2-1}{2}\right)p_2^{e_2-1}\left( \frac{p_2^2-1}{2}\right)\right)=1$. In addition, $$w(n)=\mathrm{lcm}\,\left[p_1^{e_1-1}\left( \frac{p_1^2-1}{2}\right),p_2^{e_2-1}\left( \frac{p_2^2-1}{2}\right)\right],$$ and $w(n)|p_1^{e_1-1}\left( \frac{p_1^2-1}{2}\right)p_2^{e_2-1}\left( \frac{p_2^2-1}{2}\right)$, so $$\gcd\left(k,p_1^{e_1-1}\left( \frac{p_1^2-1}{2}\right)p_2^{e_2-1}\left( \frac{p_2^2-1}{2}\right)\right)=1$$ if and only if $\gcd(k,w(n))=1$. Therefore, $D_k(x,a)$ is a permutation polynomial$\mod n$ if and only if $\gcd(k,w(n))=1$.

Suppose that the theorem is satisfied for $n$ with $r \geq 1$ different factors. Let $n=p_1^{e_1} \cdots p_r^{e_r} p_{r+1}^{e_{r+1}}$ and $m=p_1^{e_1} \cdots p_r^{e_r}$, then $\gcd(m,p_{r+1}^{e_{r+1}})=1$. Moreover, if $\gcd(a,n)=1$, then $\gcd(a,m)=1 $ and $\gcd(a,p_{r+1}^{e_{r+1}})=1$. By Lemma \ref{lemma28}, $D_k(x,a)$ is a permutation polynomial$\mod n$ if and only if $D_k(x,a)$ is a permutation polynomial$\mod m$ and$\mod p_{r+1}^{e_{r+1}}$. Of the induction hypothesis and the case $r=2$, this happens if and only if $$\gcd(k,w(m))=\gcd\left(k,p_{r+1}^{e_{r+1}-1}\left( \frac{p_{r+1}^2-1}{2}\right)\right)=1.$$

Following the same procedure as in the case $r=2$, the last equalities are equivalent to that $\gcd(k,w(n))=1$. Hence $D_k(x,a)$ is a permutation polynomial$\mod n$ if and only if $\gcd(k,w(n))=1$.

On the other hand, assume that $n$ is even, and $n=2^e p_1^{e_1}\cdots p_r^{e_r}$ where $p_i>2$ is a prime number for all $i=1, \ldots, r$. If $1\leq e <3$, let $m= p_1^{e_1}\cdots p_r^{e_r}$, then $\gcd(2^e,m)=1$ as in part a). $D_k(x,a)$ is a permutation polynomial$\mod n$ if and only if $\gcd(k,3\cdot 2^{e-1})=gcd(k,w(m))=1$, which is equivalent to $\gcd(k,w(n))=1$. If $e\geq 3$, let $p_1^{e_1}=2^e$, and the proof is identical to part in the case when $n$ is odd,since $p_1^{e_1-1}\left( \frac{p_1^2-1}{2}\right)=3\cdot2^{e-2}$. 
\end{proof}

The Dickson polynomials that are used in this work have the parameter $a = 1$; for this reason, we write $D_k (x)$ instead of $D_k (x, 1)$.

\begin{defi}
The set $D(n)$ is defined as $$D(n)=\left\lbrace D_k(x): D_k \text{ is a permutation polynomial on } \mathbb{Z}_n \right\rbrace.$$
\label{definition8}
\end{defi}

The following lemma can be found in \cite{leticia}.

\begin{lem}
$D(n)$ is an abelian semi group under the composition. 
\label{lemma9}
\end{lem}

In \cite{nobauer}, there are some results that justify the existence of the inverse of a Dickson permutation polynomial. The following definition is required:

\begin{defi}
$G_n$ is the set of permutations on $\mathbb{Z}_n$ that are induced by a Dickson polynomial; that is, if $\pi$ represents a permutation on $\mathbb{Z}_n$, then $$G_n= \left\lbrace \pi: \text{ exists } D_k(x) \in D(n) \text{ such that } \pi(a)=D_k(a) \text{ for all } a\in \mathbb{Z}_n\right\rbrace .$$
\label{definition10}
\end{defi}

It follows that $G_n$ is an abelian semi group by Lemma \ref{lemma9}. In \cite{nobauer}, it is proved that $G_n$ is an abelian group. 

Below we show some results that help us determine the cardinality of $G_{2^e}$, $G_{3^e}$ and $G_{p^e}$, $p\geq 5$, and the case $G_n$ is treated separately.

\section{$G_{2^e}$, $G_{3^e}$ and $G_{p^e}$, $p\geq 5$}
\label{sec:4}

In this section, we discuss about the epimorphism $\psi': \mathbb{Z}^*_{3 \cdot 2^{e-1}}\rightarrow G_{2^e}$ if $e<3$, $\psi'': \mathbb{Z}^*_{3 \cdot 2^{e-2}}\rightarrow G_{2^e}$ if $e\geq 3$ and $\psi: \mathbb{Z}^*_{p^{e-1}\left(\frac{p^2 -1}{2}\right)}\rightarrow G_{p^e}$ for $p>2$ a prime number, which appears in \cite{nobauer}. All these results are necessary to obtain $\mathrm{Ker}\, \psi'$, $\mathrm{Ker}\, \psi''$, $\mathrm{Ker}\, \psi$, $G_{2^e}$ and $G_{p^e}$. The proofs given here are slightly different from those that appear in \cite{nobauer}. 

\begin{teo}
The functions $\psi'$, $\psi''$ and $\psi$ as defined below are epimorphism.

\begin{tabular}{rl}
$\mathrm{1.}$ & For each $u$ in $\mathbb{Z}_{2^e}$ \\
              & \begin{tabular}{rl}
               $\mathrm{i)}$ & If $e<3$, let \\
                             & 
                             $\begin{array}{rl}
						\psi': & \mathbb{Z}^*_{3\cdot 								2^{e-1}} \rightarrow G_{2^e} \\
      					& k \mapsto D_k(u),
						\end{array}$\\
 				$\mathrm{ii)}$ & If $e\geq3$, let \\
 							& $\begin{array}{rl}
							\psi'': & \mathbb{Z}^*_{3\cdot 								2^{e-2}} \rightarrow G_{2^e} \\
      						& k \mapsto D_k(u).
							\end{array}$
 				\end{tabular}\\
$\mathrm{2.}$ & Let $p$ be an odd prime. For each $u \in \mathbb{Z}_{p^e}$, we define \\
 			  & $\begin{array}{rl}
				\psi: & \mathbb{Z}_{p^{e-1}\left( 							\frac{p^2-1}{2} \right)}^{*} \rightarrow 					G_{p^e} \\
 				& k \mapsto D_k(u).
				\end{array}$				
				\end{tabular}
\label{theorem15}
\end{teo}

\begin{proof}
The proof is given for the case $2^e$ with $e\geq3$, and the other cases are similar. Let $k,l \in \mathbb{Z}^*_{3\cdot 2^{e-2}}$ and $u \in \mathbb{Z}_{2^e}$, then by Lemma \ref{lemma5} $$ \psi'' (kl) = D_{kl}(u)=D_k\left( D_l(u) \right)=D_k(u)\circ D_l(u) = \psi''(k) \circ \psi''(l);$$ therefore, $\psi''$ is a morphism.

Let $\pi \in G_{2^e}$. A Dickson permutation polynomial $D_k(x)$ on $\mathbb{Z}_{2^e}$, which occurs if and only if $\gcd(k,3\cdot 2^{e-1})=1$. In the other words, $\gcd(k,3\cdot 2^{e-2})=1$, and so $k \in \mathbb{Z}^*_{3\cdot 2^{e-2}}$. Therefore, $\psi''$ is a epimorphism.
\end{proof}

To calculate $\mathrm{Ker}\, \psi'$ and $\mathrm{Ker}\,\psi''$, it will be shown that if $k \equiv l \pmod{3\cdot 2^{e-1}}$ or $k \equiv l \pmod{3\cdot 2^{e-2}}$, then $D_k(u) = D_l(u)$ in $\mathbb{Z}_{2^e}$ for all $u \in \mathbb{Z}_{2^e}$. For the first case, some previous results are needed, which are presented below.

\begin{lem}
Let $e\geq 4$ and $u \equiv 1 \pmod 2$, then $y^{3\cdot 2^{e-3}}=1$ in the quotient ring $\mathbb{Z}_{2^{e-1}}[y]/\langle y^2 - uy + 1 \rangle$.
\label{lemma11}
\end{lem}

\begin{proof}
Since $u \equiv 1 \pmod 2$, $u=2t+1$ with $t \in \mathbb{Z}$. We use induction on $e$.

For $e=4$, we have $y^2=uy-1$ in $\mathbb{Z}_{2^3}/\langle y^2 - uy + 1 \rangle$, then $$y^3 = y(uy-1)=u(uy-1)-y=y(u^2-1) -u $$ and so 
\begin{center}
\begin{equation}
\begin{array}{rl}
y^6 = & (y^3)^2 = y^2(u^4-2u^2+1)+y(-2u^3+2u)+u^2 \\
    = & (uy-1)(u^2-1)^2-2uy(u^2-1)+u^2 \\
    = & uy(u^2-1)(u^2-3)-(u^2-1)^2+u^2.
\end{array}
\label{eq5}
\end{equation}
\end{center}

Since $u^2-1=4(t^2+t)$ and $u^2-3 \equiv 0 \pmod 2$, it follows that 
\begin{equation}
u(u^2-1)(u^2-3)\equiv 0 \pmod {2^3} .
\label{eq6}
\end{equation}

Moreover, $t^2 + t$ is even for all $t$ in $\mathbb{Z}$, then
\begin{equation}
-(u^2-1)^2+u^2=4(t^2 + t)+1 \equiv 1 \pmod {2^3} .
\label{eq7}
\end{equation}

Therefore, $y^6=1$ in $\mathbb{Z}_{2^3}/\langle y^2 - uy + 1 \rangle$.

For $e=k+1$, we have $y^{3\cdot 2^{e-1}}= \left( y^{3\cdot 2^{e-2}} \right)^2$, by the induction hypothesis $y^{3\cdot 2^{e-2}}=1$ in the quotient ring $\mathbb{Z}_{2^{e-2}}/\langle y^2 - uy + 1 \rangle$, that is $y^{3\cdot 2^{e-2}}=1 + 2^{e-2}v(y)$ in $\mathbb{Z}[y]/\langle y^2 - uy + 1 \rangle$ with $v(y) \in \mathbb{Z}[y]$. Therefore, $$\left( y^{3\cdot 2^{e-2}} \right)^2 = \left( 1+ 2^{e-2}v(y)  \right)^2 = 1+2^{e-1}\left( v(y) + 2^{e-1}v^2(y) \right)=1,$$ thus, $y^{3\cdot 2^{e-1}} \in \mathbb{Z}_{2^{e-1}}/\langle y^2 - uy + 1 \rangle.$

\end{proof}

\begin{cor}
Let $e\geq 4$, then $y^{3\cdot 2^{e-3}}=1+2^{e-1}(18y-7)$ in the quotient ring $\mathbb{Z}_{2^e}[y]/\langle y^2 - 3y + 1 \rangle$.
\label{corollary12}
\end{cor}

\begin{proof}
The proof is similar to the proof of Lemma \ref{lemma11}, considering that for $u=3$ and (\ref{eq5}) $u(u^2-1)(u^2-3)=3(8)(6)=8(18)$, and $u^2-(u^2-1)^2=9-8^2=1+8(-7)$. Therefore, $$y^6 = 1 +8(18y-7) \in \mathbb{Z}_{2^4}[y]/\langle y^2 - 3y + 1 \rangle.$$
\end{proof}

\begin{lem}
Let $e\geq 3$ and $u \equiv 0 \pmod 4$, then $y^{2^{e-1}}=1$ in the quotient ring $\mathbb{Z}_{2^e}[y]/\langle y^2 - uy + 1 \rangle$.
\label{lemma13}
\end{lem}

\begin{proof}
Since $u \equiv 0 \pmod 4$, then $u=2^2\cdot t$ for some $t \in \mathbb{Z}$. Note that $y^2=uy-1$ in $\mathbb{Z}_{2^3}[y]/\langle y^2 - uy + 1 \rangle$ and 
\begin{equation*}
\begin{array}{rl}
y^4 = & (uy-1)^2 = u^3y-2uy-u^2+1 = 2^6t^3y- 2^3ty-2^4t^2+1 \\
    = & 1+2^3 \left( 2^3t^3y-ty-2t^2 \right)=1 \text{ in } \mathbb{Z}_{2^3}[y]/\langle y^2 - uy + 1 \rangle. 
\end{array}
\end{equation*}

The rest of the proof is analogous to proof of Lemma \ref{lemma11}.
\end{proof}

\begin{lem}
Let $e\geq 2$ and $u = 2+4g$ with $g \in \mathbb{Z}$, then $\eta^{2^{e-1}}=1$ in $\mathbb{Z}_{2^e}[y]/\langle (y-1)^2-2^{e-2}gy , 2^{e-2}(y-1) \rangle$, where $\eta$ is the solution of $y^2 - uy +1 =0$.
\label{lemma14}
\end{lem}

\begin{proof}
Since $u=2+4g$, then $y^2 = uy-1 = (2+4g)y-1 = 1+2(y-1)+4gy=1$ in $\mathbb{Z}_{2^2}[y]/\langle (y-1)^2-gy ,(y-1) \rangle$, also $y^2-uy+1= y^2-(2+4g)y+1=(y-1)^2 - 2^2gy$ in $\mathbb{Z}_{2^2}[y]$.

The rest of the proof follows by induction on $e$, as in Lemma \ref{lemma11}.
\end{proof}

\begin{cor}
Let $e\geq 3$, $k,l$ be positive integers such that $k \equiv l \pmod{3 \cdot 2^{e-1}}$ and $u \in \mathbb{Z}_{2^e}$, then the equation $y^2 -uy +1 =0$ has a solution $\eta$ in an extension of the ring $\mathbb{Z}_{2^e}$ such that:
\begin{itemize}
\item [$\mathrm{1.}$] If $u\equiv 1 \pmod 2$, $\eta^{3\cdot 2^{e-2}} = 1$ and $D_k(u) = D_l(u)$ in $\mathbb{Z}_{2^e}$.
\item [$\mathrm{2.}$] If $u\equiv 0 \pmod 2$, $\eta^{2^{e-1}} = 1$ and $D_k(u) = D_l(u)$ in $\mathbb{Z}_{2^e}$.
\end{itemize}

\label{corollary16}
\end{cor}

\begin{proof}
Since $k \equiv l \pmod{3 \cdot 2^{e-1}}$, then $k=l+ 3 \cdot 2^{e-1} \cdot t$ for some $t \in \mathbb{Z}$. To prove the first part of this corollary, we have that $e+1 \geq 4$, and by Lemma \ref{lemma11} there exists a solution $\eta$ of $y^2-uy+1=0$ in $\mathbb{Z}_{2^e}[y]/\langle y^2-uy+1 \rangle$ such that $\eta^{3\cdot 2^{e-2}} = 1$. In addition, $u=\eta + \eta^{-1}$, and by (\ref{eq4})
\begin{equation}
\begin{array}{rl}
D_k(u) = & D_k(\eta + \eta^{-1}) = \eta^k + \eta^{-k} = \eta^{l+ 3 \cdot 2^{e-1}} + \eta^{-(l+ 3 \cdot 2^{e-1})} \\
    = & \eta^l\left( \eta^{3 \cdot 2^{e-2}} \right)^{2t}+\eta^{-l}\left( \eta^{3 \cdot 2^{e-2}} \right)^{-2t} = \eta^l + \eta^{-l} \\
    = & D_l(\eta + \eta^{-1})=D_l(u).
\end{array}
\label{eq8}
\end{equation}

To prove the second part of this corollary, we considerer two cases. The first case, if $u \equiv 0 \pmod 4$, by Lemma \ref{lemma13} $y^2 -uy + 1=0$ has a solution $\eta$ in $\mathbb{Z}_{2^e}[y]/\langle y^2-uy+1 \rangle$ such that $\eta^{2^{e-1}} = 1$. The second case, if $u \equiv 2 \pmod 4$, by Lemma \ref{lemma14} there exists a solution $\eta$ of $y^2 -uy + 1=0$ in the ring $\mathbb{Z}_{2^e}[y]/\langle (y-1)^2-2^{e-2}gy , 2^{e-2}(y-1) \rangle$ such that $\eta^{2^{e-1}} = 1$. Analogous to (\ref{eq8}), we have $D_k(u)=D_l(u)$. \end{proof}

\begin{pro}
Let $k,l$ be positive integers such that $k \equiv l \pmod{3 \cdot 2^{e-1}}$, then for all $u \in \mathbb{Z}_{2^e}$, $D_k(u) = D_l(u)$ in $\mathbb{Z}_{2^e}$.
\label{Proposition17}
\end{pro}

\begin{proof}
Let $u \in \mathbb{Z}_{2^e}$. Since $k \equiv l \pmod{3 \cdot 2^{e-1}}$, then $k=l+ 3 \cdot 2^{e-1} \cdot t$ for some $t \in \mathbb{Z}$.

The proof is given in several cases. For the first case, if $e\geq 3$, the proposition is satisfied by Corollary \ref{corollary16}.

For the second case, if $e=1$ and $k \equiv l \pmod 3$, then $k=l+3t$ for some $t \in \mathbb{Z}$; also the equation $y^2-uy+1=0$ has a solution $\eta$ in $\mathbb{Z}_2 [y]/\langle y^2 - uy +1 \rangle$ and satisfies $\eta^{2^2 -1}=\eta^{3}=1$, as in (\ref{eq8}) $D_k(u)=D_l(u)$.

For the last case, if $e=2$ by (\ref{eq5}), (\ref{eq6}) and (\ref{eq7}) of the proof of Lemma \ref{lemma11} and since $k \equiv l \pmod 6$, $k=l+6t$ for some $t \in \mathbb{Z}$: $$y^6 = 8(2t+1)(t^2+t)(2t^2+2t-1)y + 4t^2 + 4t+1 = 1$$ in $\mathbb{Z}_4 [y]/\langle y^2 - uy +1 \rangle$ with $u \equiv 1 \pmod 2$. Let $\eta$ be a solution of $y^2-uy+1=0$, then $\eta^6=1$ and as in (\ref{eq8}) $D_k(u)=D_l(u)$. Now if $u \equiv 0 \pmod 2$, we have two cases: $u\equiv 2 \pmod 4$ or $u \equiv 0 \pmod 4$. For the first, by Lemma \ref{lemma14} $y^2 -uy +1 =0$ has a solution $\eta$ such that $\eta^{2}=1$; thus, $\eta^6 = 1$ and as in (\ref{eq8}) $D_k(u)=D_l(u)$. In addition, for the second case, $y^2 = uy-1 = -1$ in $\mathbb{Z}_4 [y]/\langle y^2 - uy +1 \rangle$, so $y^6=(y^2)^3=-1$ and $y^{12}=1$, also $u=\eta+\eta^{-1}$ where $\eta$ is solution of $y^2-uy+1=0$. We have $$D_k(u)=D_k(\eta+\eta^{-1})=\eta^k+\eta^{-k}=\eta^{l+6t}+\frac{1}{\eta^{l+6t}}=\dfrac{\eta^{2l}+1}{\eta^{l+6t}},$$ then $$D_k(u)-D_l(u)=\dfrac{\eta^{2l}+1}{\eta^{l+6t}}-\eta^l-\frac{1}{\eta^l}=\dfrac{(\eta^{2l}+1)(1-\eta^{6t})}{\eta^{l+6t}}.$$

If $t$ is even, $\eta^{6t}=1$ and so $D_k(u)=D_l(u)$. Otherwise, $\eta^{6t}=\eta^6 =-1$, and if $l$ is even, $\eta^{2l}+1=2$. It then follows that $(\eta^{2l}+1)(1-\eta^6)=0$ in $\mathbb{Z}_4 [y]/\langle y^2 - uy +1 \rangle$ or if $l$ is odd, $\eta^{2l}=\eta^2=-1$; thus, $\eta^{2l}+1=0$ in $\mathbb{Z}_4 [y]/\langle y^2 - uy +1 \rangle$. Therefore, $D_k(u)=D_l(u)$. \end{proof}

\begin{lem}
Let $e\geq 3$, then $D_{3\cdot 2^{e-2}+1}(u)=u$ in $\mathbb{Z}_{2^e}$ for all $u \in \mathbb{Z}_{2^e}$.
\label{lemma18}
\end{lem}

\begin{proof}
Let $u \in \mathbb{Z}_{2^e}$, and $k=3\cdot 2^{e-2}+1$. We consider the following cases for $u$: $u \equiv 1 \pmod 2$ or $u \equiv 0 \pmod 2$.

If $u \equiv 1 \pmod 2$, by Corollary \ref{corollary16}, the equation $y^2 -uy +1 =0$ has a solution $\eta$ in an extension of the ring $\mathbb{Z}_{2^e}$ such that $\eta^{3\cdot2^{e-2}}=1$; also $u = \eta + \eta^{-1}$, and therefore $$D_k(u)=D_k(\eta + \eta^{-1})=\eta^k + \eta^{-k}=\eta^{3\cdot 2^{e-2}}\eta + \eta^{-3\cdot 2^{e-2}}\eta^{-1}=\eta + \eta^{-1} = u.$$

Assume that $u \equiv 0 \pmod 2$, by Corollary \ref{corollary16}, the equation $y^2 -uy +1 =0$ has a solution $\eta$ in an extension of the ring $\mathbb{Z}_{2^e}$ such that $\eta^{2^{e-1}}=1$ and $$\eta^k=\eta^{3\cdot 2^{e-2}}\eta = \eta^{2^{e-1}}\eta^{2^{e-2}}\eta=\eta^{2^{e-2}}\eta.$$

Thus 
\begin{equation*}
\begin{array}{rl}
D_k(u) - u = & \eta^k + \eta^{-k} - \eta - \eta^{-1} = \eta^{2^{e-2}}\eta + \frac{1}{\eta^{2^{e-2}}\eta} - \eta - \frac{1}{\eta} \\
    = & \dfrac{\eta^{2l}\eta^2 + 1 - \eta^{2^{e-2}}\eta^{2}-\eta^{2^{e-2}}}{\eta^{2^{e-2}}\eta} = \dfrac{(\eta^{2}+1)(1-\eta^{2^{e-2}})}{\eta^{2^{e-2}}\eta}.
\end{array}
\end{equation*}

Again we consider two cases: $u \equiv 0 \pmod 4$ or $u \equiv 2 \pmod 4$. 

Assume that $u \equiv 0 \pmod 4$. For $e>3$, Lemma \ref{lemma13} is satisfied for $e-1 \geq 3$, and thus $\eta^{2^{e-2}}=1$ in $\mathbb{Z}_{2^{e-1}}[y]/\langle y^2 -uy +1 \rangle$; as a result, $1-\eta^{2^{e-2}}=2^{e-1}v(\eta)$ in $\mathbb{Z}[y]/\langle y^2 -uy +1 \rangle$. Thus $$(\eta^2+1)(1-\eta^{2^{e-2}})=2^{e-1}u\eta v(\eta) = 2^e (2t\eta v(\eta))$$ for some $t \in \mathbb{Z}$. Therefore, $(\eta^2+1)(1-\eta^{2^{e-2}})=0$, and $D_k(u)=u$ in $\mathbb{Z}_{2^e}$. Now for $e=3$, $\eta^4=1$ and $(\eta^2+1)(1-\eta^2)=1-\eta^4=0$; therefore $D_k(u)=u$.

If $u \equiv 2 \pmod 4$, then $u=2+4g$ for some $g \in \mathbb{Z}$ and $\eta^2=u\eta - 1$. Therefore, $$\eta^2 + 1 = u \eta = 2(1+2g)\eta.$$ 

On the other hand, $e-1 \geq 2$, and by Lemma \ref{lemma14}, $\eta^{2^{e-2}}=1$ in the quotient ring $\mathbb{Z}_{2^{e-1}}[y]/\langle (y-1)^2-2^{e-3}gy,2^{e-3}(y-1) \rangle$; that is $1-\eta^{2^{e-2}}=2^{e-1}v(\eta)$ in $\mathbb{Z}[y]/\langle (y-1)^2-2^{e-3}gy,2^{e-3}(y-1) \rangle$. Thus, $(\eta^2+1)(1-\eta^{2^{e-2}})=0$, and therefore $D_k(u)-u=0$ in $\mathbb{Z}_{2^e}$.\end{proof}

\bigskip

We define the epimorphism $\bar{\psi}$ for $\mathbb{Z}_{2^e}$ as

\begin{equation*}
\bar{\psi}=\left\lbrace \begin{array}{ll}
\psi', & \text{if } e<3, \\
\psi'', & \text{if } e \geq 3.
\end{array}\right.
\end{equation*}

By notation, let $K_{2^e}=\mathrm{Ker}\, \bar{\psi}$, and we have the following theorem.

\begin{teo}
We consider the epimorphism $\bar{\psi}$, then $K_{2^e}=\left\lbrace 1,-1\right\rbrace$.
\label{theorem19}
\end{teo}

\begin{proof}
If $\eta$ is a solution of the equation $y^2-uy+1=0$ for some $u \in \mathbb{Z}_{2^e}$, then $u=\eta+\eta^{-1}$; note that 
\begin{equation}
D_1(u)=\eta+\eta^{-1}=u \text{ and } D_{-1}(u)=\eta^{-1}+\eta=u.
\label{eq9}
\end{equation} 

Therefore, $\left\lbrace 1,-1 \right\rbrace \subseteq \mathrm{Ker}\,\bar{\psi}$. In order to prove that $\mathrm{Ker}\, \bar{\psi}\subseteq \left\lbrace 1,-1 \right\rbrace$, we consider the following cases: $e=1$, $e=2$ and $e\geq 3$.

If $e=1$, let $k \in \mathrm{Ker}\,\bar{\psi} \subseteq \mathbb{Z}^*_3 = \left\lbrace 1,2 \right\rbrace$. Note that $2=-1$ in $\mathbb{Z}^*_3$, and by (\ref{eq9}) $\mathrm{Ker}\,\bar{\psi} = \left\lbrace 1,-1 \right\rbrace$.

The case $e=2$ is similar to the previous case noting that $\mathrm{Ker}\,\bar{\psi} \subseteq \mathbb{Z}^*_6 = \left\lbrace 1,5 \right\rbrace$.

Assume that $e \geq 3$. Let $K_e=\mathrm{Ker}\,\bar{\psi}\subseteq \mathbb{Z}^*_{3\cdot2^{e-2}}$, by Proposition \ref{Proposition17} if $k\equiv 1 \pmod{3\cdot2^{e-1}}$; that is $k \in K_{e+1}$, then $D_{k}(u)=u$ in $\mathbb{Z}_{2^e}$. The same happens if $k\equiv 1 \pmod{3\cdot2^{e-1}}$; in this way, $k\equiv \pm 1 \pmod{3\cdot2^{e-2}}$, so $k= \pm 1 + 3\cdot2^{e-2}t$ for some $t \in \mathbb{Z}$. If $t$ is even, $k\equiv \pm 1 \pmod{3\cdot2^{e-1}}$; otherwise, $k\equiv \pm 1 + 3\cdot2^{e-2} \pmod{3\cdot2^{e-1}}$. In addition,
\begin{equation*}
\begin{array}{rl}
-1 + 3\cdot2^{e-2} \equiv & -1 + 3\cdot2^{e-2} - 3\cdot2^{e-1} \pmod {3\cdot2^{e-1}} \\
    \equiv & -(1+3\cdot2^{e-2}) \pmod{3\cdot2^{e-1}},
\end{array}
\end{equation*}
so it enough to show that $k = 1 + 3\cdot2^{e-2} \notin K_{e+1}$, in other words we have to prove that $k= 1 + 3\cdot2^{e-3} \notin K_{e}$ with $e\geq4$.

Assume $k \in K_e$. For $u = 3$, the equation $y^2-uy + 1 = 0$, has a solution $\eta$ in the ring $\mathbb{Z}_{2^e}[y]/ \langle y^2-3y+1 \rangle$, which satisfies $u=\eta+\eta^{-1}$ and $D_k(u)=u$; that is $\eta^k + \frac{1}{\eta^k}=\eta + \frac{1}{\eta}$, which is equivalent to $\left( \eta^{k-1}-1\right)\left( \eta^{k+1}-1\right)=0$. On the other hand, $k-1=3\cdot2^{e-3} \text{ and } k+1 = 2+3\cdot2^{e-3}$, so that $\left( \eta^{3\cdot2^{e-3}}-1\right)\left( \eta^{2+3\cdot2^{e-3}}-1\right)=0$. In addition, by Corollary \ref{corollary12} $\eta^{3\cdot2^{e-3}}=1+2^{e-1}(18\eta-7)$ in $\mathbb{Z}_{2^e}[y]/ \langle y^2-3y+1 \rangle$, then 
\begin{equation*}
\begin{array}{rl}
\left( \eta^{3\cdot2^{e-3}}-1\right)\left( \eta^{2+3\cdot2^{e-3}}-1\right) = &  2^{e-1}(18\eta-7)\left[ \eta^2(1+2^{e-1}(18\eta-7)) +1 \right] \\
 = &  -2^{e-1}\eta^2 \left[1+2^{e-1}(18\eta-7) \right] - 2^{e-1} \\
= & -2^{e-1}\eta^2(1-2^{e-1})-2^{e-1} \\
   = &  -2^{e-1}\eta^2-2^{e-1} = -2^{e-1}(\eta^2+1)\\
   = &  -2^{e-1}(3\eta) = -2^{e-1}\eta=0,
\end{array}
\end{equation*}
which is not possible. Therefore $k\notin K_e$, and so $\mathrm{Ker}\,\bar{\psi}\subseteq \left\lbrace 1,-1 \right\rbrace$. \end{proof}

Now we give a list of the theorems that we will use in the next section. The results that interest us are those that are similar to Proposition \ref{Proposition17} and Theorem \ref{theorem19} when $p$ is an odd prime. The details of the proofs can be found in \cite{nobauer}.

\begin{teo}
Let $p$ be a odd prime, $e \geq 1$ and $k,l \in \mathbb{N}$ such that $k \equiv l \pmod{p^{e-1}\left( \frac{p^2-1}{2} \right)}$. Thus, for all $u \in \mathbb{Z}_{p^e}$ $$D_k(u)=D_l(u) \text{ in } \mathbb{Z}_{p^e}.$$
\label{theorem23}
\end{teo}

\begin{teo}
Let $p \geq 5$ be a prime and $\psi$ the epimorphism of Theorem \ref{theorem15}, then the following holds:
\begin{itemize}
\item [$\mathrm{1.}$]If $e=1$. Then $\mathrm{Ker}\,(\psi)=\left\lbrace 1,-1,p,-p \right\rbrace$.
\item [$\mathrm{2.}$]If $e>1$. Then $\mathrm{Ker}\,(\psi)=\left\lbrace 1,-1 \right\rbrace$.
\end{itemize}
\label{Proposition25}
\end{teo}

\begin{teo}
Let $p=3$, then the kernel of $\psi$, where $\psi$ is the epimorphism of Theorem \ref{theorem15}, is $\mathrm{\mathrm{Ker}\,}(\psi)=\lbrace 1,-1 \rbrace$.
\label{theorem31}
\end{teo}
\bigskip

Now let $K_{3^e}=\mathrm{Ker}\, \psi$ for $p=3$ and $K_{p^e}=\mathrm{Ker}\, \psi $ for $p\geq 5$ in Theorem \ref{theorem15}, and by Theorems \ref{theorem19}, \ref{Proposition25} and \ref{theorem31}, we have
\begin{equation}
K_{2^e}=K_{3^e}=\lbrace 1,-1\rbrace,
\label{eq35}
\end{equation}
and for $p\geq 5$
\begin{equation}
K_{p^e}=\left\lbrace \begin{array}{ll}
\lbrace 1,-1 \rbrace, & \text{if } e>1, \\
\lbrace 1,-1,p,-p\rbrace, & \text{if } e=1.
\end{array}\right.
\label{eq36}
\end{equation}

By the first Isomorphism Theorem and Theorem \ref{theorem15}, we obtain
\begin{equation*}
\begin{array}{rl}
G_{2^e} \simeq & \frac{\mathbb{Z}^*_{l_0}}{K_{2^e}}, \\
\vspace{0.1cm} & \\
G_{3^e} \simeq & \frac{\mathbb{Z}^*_{4 \cdot 3^{e-1}}}{K_{3^e}}, \\
\vspace{0.1cm} & \\
G_{p^e} \simeq & \frac{\mathbb{Z}^*_{p^{e-1}\left(\frac{p^2 -1}{2}\right)}}{K_{p^e}}. 
\end{array}
\label{eq37}
\end{equation*}

Therefore, by using Euler's phi function $\varphi$, we have the following values for the cardinality of $G_{p^e}$: $|G_{2^e}|=\varphi(l_0)/2$, $|G_{3^e}|=\varphi(4\cdot 3^{e-1})/2$ and $|G_{p^e}|=\varphi\left(p^{e-1}\left(\frac{p^2-1}{2}\right)\right)/|K_{p^e}|$; see \cite[Ch3]{leticia}. Thus,
\begin{equation*}
|G_{p^e}|= \left\lbrace 
\begin{array}{ll}
1, & \text{if } p=2 \text{ and } e<3, \\
2^{e-3}, & \text{if } p=2 \text{ and } e\geq3, \\
1, & \text{if } p=3 \text{ and } e=1 ,\\
2\cdot 3^{e-2}, & \text{if } p=3 \text{ and } e>1, \\
\frac{1}{4}\cdot\varphi\left( \frac{p^2-1}{2}\right) , & \text{if } p\geq5 \text{ and } e=1,\\
\frac{p^{e-2}(p-1)}{2}\cdot\varphi\left( \frac{p^2-1}{2}\right) , & \text{if } p\geq5 \text{ and } e>1.
\end{array}\right.
\end{equation*}

\section{$G_n$}
\label{sec:5}

In this section following the idea of Section 4, an epimorphism $\xi: \mathbb{Z}_{w(n)}^{*} \rightarrow G_n$ is defined. In this part is proved the main result of this work. First, we obtain $K_n : = \mathrm{Ker}\, \xi$. To compute $|K_n|$, we give a bijection $\rho: A \rightarrow K_n$. 

Before proceeding with results for $G_n$, some previous results are shown. The following proposition is a particular case of a result that appears in \cite[Ch.2]{niven}, we develop the prove because it will be useful in later results.

\begin{pro}
Let $a,b,p,q \in \mathbb{Z}$ and $d:=\gcd (p,q)$, $l:=\mathrm{\mathrm{lcm}\,}[p,q]$. The congruence system 
\begin{equation}
x \equiv a \pmod p
\label{eq12}
\end{equation}
\begin{equation}
x \equiv b \pmod q
\label{eq13}
\end{equation}
then has a solution if and only if $d|(a-b)$. Moreover, if $x_1,x_2$ are solutions of the congruence system, then $x_1 \equiv x_2 \pmod l$.
\label{Proposition29}
\end{pro}

\begin{proof}
Let $x=a+ps_1$ be a solution of (\ref{eq12}), where $s_1 \in \mathbb{Z}$. We have to show that there exists $s_1 \in \mathbb{Z}$ such that $$a+ps_1 \equiv b \pmod q .$$

Since $d=\gcd (a,b)$, there exist $x,y \in \mathbb{Z}$ such that $d=px+qy$. In addition, $d|(a-b)$, then $$a-b=(px+qy)t,$$ for some $t \in \mathbb{Z}$. Therefore, $$q | (a-b+p(-xt));$$ let $s_1=-xt$, so $a+p(-xt) \equiv b \pmod q $.

On the other hand, if $x \in \mathbb{Z}$ is a solution of the congruence system, then $x= a+ps_1 = b+qs_2$ for $s_1,s_2 \in \mathbb{Z}$. Thus $a-b=ps_1 + q(-s_2)$, since $d=\gcd(a,b)$, then $d|(ps_1+q(-s_2))$, and therefore $d|(a-b)$.

Let $x_1,x_2$ be solutions of the congruence system, then 
$x_1 \equiv a \pmod p$, $x_1 \equiv b \pmod q $, $x_2 \equiv a \pmod p$ and $x_2 \equiv b \pmod q$. Thus, $x_1 \equiv x_2 \pmod p$ and $x_1 \equiv x_2 \pmod q$; therefore $x_1 \equiv x_2 \pmod l$.
\end{proof}

As a consequence of Proposition \ref{Proposition29}, let $l_i = p_i^{e_i -1} \left( \frac{p^2_i -1}{2} \right)$ for $i=1,\ldots,r$ with $p_i \geq 3$ a prime number, and 
\begin{equation*}
l_{0} = \left\{
\begin{array}{ll}
3 \cdot 2^{e-1} & \text{, if } 1 \leq e<3,\\
\vspace{0.1cm} & \\
3 \cdot 2^{e-2} & \text{, if } e\geq3.
\end{array} \right.
\end{equation*}

For $a \in K_{p_i^{e_i}}$ and $b \in K_{p_j^{e_j}}$, where $p_i \neq p_j$ and $p_i , p_j \in \left\lbrace 2^e,p_1^{e_1}, \ldots, p_r^{e_r}\right\rbrace $, then the congruence system 
\begin{equation*}
\begin{array}{c}
x \equiv a \pmod {l_i} \\
x \equiv b \pmod {l_j}
\end{array},
\end{equation*}
with $i,j \in \left\lbrace 0,1, \ldots, r \right\rbrace $ and $i\neq j$ has solution if and only if $\gcd (l_i,l_j)|(a-b)$.

\begin{pro}
Let $k,l \in \mathbb{N}$ such that $k \equiv l \pmod {w(n)}$. Thus, $D_k(u)=D_l(u)$ in $\mathbb{Z}_n$ for all $u \in \mathbb{Z}_n$.
\label{Proposition31}
\end{pro}

\begin{proof}
Since $k \equiv l \pmod {w(n)}$, then $k \equiv l \pmod {l_i}$ for $i=1,\ldots,r$. Furthermore, suppose that $e>0$ so $k \equiv l \pmod {l_0}$. The case $e=0$ is omitted in the proof.

Let $u \in \mathbb{Z}_n$, by Proposition \ref{Proposition17} and Theorem \ref{theorem23} $D_k(u)=D_l(u)$ in $\mathbb{Z}_{2^e}$ and $\mathbb{Z}_{p_i^{e_i}}$ for each $i=1,\ldots,r$. By (\ref{eq40}) $\gcd\left( 2^e,p_1^{e_1},\ldots,p_r^{e_r} \right)=1$, then $D_k(u)=D_l(u)$ in $\mathbb{Z}_n$.
\end{proof}

A morphism similar to those in Theorem \ref{theorem15} is defined.

\begin{pro}
For each $u \in \mathbb{Z}_n$ the function $\xi$ is defined as
\begin{equation*}
\begin{array}{rl}
\xi: & \mathbb{Z}_{w(n)}^{*} \rightarrow G_{n} \\
 & k \mapsto D_k(u),
\end{array}
\end{equation*}
then $\xi$ is an epimorphism.
\label{Proposition32}
\end{pro}

\begin{proof}
The proof is identical to the proof of Theorem \ref{theorem15}.
\end{proof}

Let $K_n = \mathrm{\mathrm{Ker}\,}(\xi)$ be the kernel of the epimorphism $\xi$. Recall that from the previous sections we have the following:

$$K_{2^e}=\lbrace 1,-1 \rbrace \subseteq \mathbb{Z}^{*}_{l_0},$$

$$K_{3^e}=\lbrace 1,-1 \rbrace \subseteq \mathbb{Z}^{*}_{2^2\cdot3^{e-1}}.$$

For $p_i \geq 5$, where $p_i$ is a prime number,
\begin{equation*}
K_{p_i^{e_i}} = \left\{
\begin{array}{ll}
\lbrace 1,-1 \rbrace,  & \text{if } e_i >1,\\
\vspace{0.1cm} & \\
\lbrace 1,-1,p_i , -p_i \rbrace,  & \text{if } e_i =1.
\end{array} \right.
\end{equation*}

Thus, $K_{p_i^{e_i}} \subseteq \mathbb{Z}^{*}_{l_i}$.

\begin{lem}
Let $n=p_1^{e_1} \cdots p_r^{e_r}$ be a product of different odd prime numbers. Thus, $k \in K_n$ if and only if $a_i \in K_{p_i^{e_i}}$ exists for all $i=1, \ldots, r$ such that $k$ is a solution of the congruence system:
\begin{equation}
\begin{array}{rl}
x \equiv & a_1 \pmod {l_1} \\
\vdots & \\
x \equiv & a_r \pmod {l_r}.
\end{array}
\label{eq15}
\end{equation}
\label{lemma33}
\end{lem}

\begin{proof}
$k \in K_n$ if and only if $D_k(u)=u$ in $\mathbb{Z}_n$ for all $u \in \mathbb{Z}_n$, since
$p_1,\ldots,p_r$ are relative primes in pairs and by the Chinese Remainder Theorem this happens if and only if $D_k(u)=u \in \mathbb{Z}_{p_i^{e_i}}$ for all $i=1,\ldots,r$, which is equivalent to there existing $a_i \in K_{p_i^{e_i}}$ for each $i=1,\ldots,r$ such that 
\begin{equation*}
\begin{array}{rl}
k \equiv & a_1 \pmod {l_1} \\
\vdots & \\
k \equiv & a_r \pmod {l_r}.
\end{array}
\end{equation*}
That is $k$ is a solution of the congruence system (\ref{eq15}).
\end{proof}

\begin{lem}
Let $n=2^e \cdot p_1^{e_1} \cdots p_r^{e_r}$ where $p_1, \ldots, p_r$ are odd prime numbers. Thus, $k \in K_n$ if and only if $a_0 \in K_{2^e}$ and $a_i \in K_{p_i^{e_i}}$ exists for all $i=1, \ldots, r$ such that $k$ is a solution of the congruence system:
\begin{equation}
\begin{array}{rl}
x \equiv & a_0 \pmod {l_0} \\
x \equiv & a_1 \pmod {l_1} \\
\vdots & \\
x \equiv & a_r \pmod {l_r}.
\end{array}
\label{eq16}
\end{equation}
\label{lemma34}
\end{lem}

\begin{proof}
$k \in K_n$ if and only if $D_k(u)=u$ in $\mathbb{Z}_n$ for all $u \in \mathbb{Z}_n$, since
$2,p_1,\ldots,p_r$ are relative primes in pairs. By Chinese Remainder Theorem, this happens if and only if $D_k(u)=u \in \mathbb{Z}_{2^e}$ and $D_k(u)=u \in \mathbb{Z}_{p_i^{e_i}}$ for all $i=1,\ldots,r$, which is equivalent to the existence of $a_0 \in K_{2^e}$ and $a_i \in K_{p_i^{e_i}}$ for each $i=1,\ldots,r$ such that 
\begin{equation*}
\begin{array}{rl}
k \equiv & a_0 \pmod {l_0} \\
k \equiv & a_1 \pmod {l_1} \\
\vdots & \\
k \equiv & a_r \pmod {l_r}.
\end{array}
\end{equation*}
In other words, $k$ is a solution of the congruence system (\ref{eq16}).
\end{proof}

\begin{rk}
Let $b_1, \ldots, b_t \in \mathbb{N}$ and $m=\mathrm{lcm}\, [b_1, \ldots,b_t]$; we define 
\begin{equation*}
\begin{array}{rrcl}
\mu: & \mathbb{Z}_m^{*} & \rightarrow & \mathbb{Z}_{b_1}^{*} \times \cdots \times \mathbb{Z}_{b_t}^{*} \\
& k \mod m & \mapsto & (k \mod b_1 , \ldots, k \mod b_t),
\end{array}
\end{equation*}
then $\mu$ is a monomorphism.
\label{observation35}
\end{rk}

\begin{proof}
$\mu$ is well-defined because $k \equiv l \pmod m$ implies that $k \equiv l \pmod {b_i}$ for all $i=1,\ldots,t$, so $(k \mod b_1 , \ldots, k \mod b_t)=(l \mod b_1 , \ldots, l \mod b_t)$.

Since $kl \mod l_i = (k \mod l_i)(l \mod l_i)$ for each $i=1,\ldots,t$ then
\begin{equation*}
\begin{array}{rl}
\mu(kl)=& (kl \mod b_1 , \ldots, kl \mod b_t) \\
=& (k \mod b_1 , \ldots, k \mod b_t)(l \mod b_1 , \ldots, l \mod b_t) \\
= & \mu(k)\mu(l).
\end{array}
\end{equation*}

We note that the identity in $\mathbb{Z}_{b_1}^{*} \times \cdots \times \mathbb{Z}_{b_t}^{*}$ is $(1 \mod b_1 , \ldots, 1 \mod b_t)$, so that if $k \in \mathrm{\mathrm{Ker}\,}(\mu)$, then $\mu(k)=(1 \mod b_1 , \ldots, 1 \mod b_t)$, that is $k \equiv 1 \pmod {b_i}$ for all $i=1,\ldots,t$. Therefore $k \equiv 1 \pmod m$ and so $\mathrm{\mathrm{Ker}\,}(\mu)= \lbrace 1 \rbrace$, and $\mu$ is a monomorphism.
\end{proof}

For all $i=1,\ldots,r$, each $K_{p_i^{e_i}}$ is subgroup of $\mathbb{Z}^{*}_{p_i^{e_i}\left( \frac{p_i^2-1}{2}\right)}=\mathbb{Z}^{*}_{l_i}$, and $K_{2^e}$ is subgroup of $\mathbb{Z}^{*}_{l_0}$; therefore, $K_{p_1^{e_1}}\times \cdots \times K_{p_r^{e_r}}$ is subgroup of $\mathbb{Z}^{*}_{l_1} \times \cdots \times \mathbb{Z}^{*}_{l_r}$, and $K_{2^e}\times K_{p_1^{e_1}}\times \cdots \times K_{p_r^{e_r}}$ is subgroup of $\mathbb{Z}^{*}_{l_0} \times \mathbb{Z}^{*}_{l_1} \times \cdots \times \mathbb{Z}^{*}_{l_r}$.

Let $A_1:=\lbrace(a_1,\ldots ,a_r) \in K_{p_1^{e_1}}\times \cdots \times K_{p_r^{e_r}} : \text{ (\ref{eq15}) has a solution}\rbrace$, and $$A_0:=\lbrace(a_0,a_1,\ldots ,a_r) \in K_{2^e}\times K_{p_1^{e_1}}\times \cdots \times K_{p_r^{e_r}} : \text{ (\ref{eq16}) has a solution}\rbrace.$$

We define
\begin{equation}
A: = \left\{
\begin{array}{ll}
A_0  & \text{, if } e\neq 0,\\
\vspace{0.1cm} & \\
A_1  & \text{, if } e=0.
\end{array} \right.
\label{eq17}
\end{equation}

Note that $A$ is a subset of $K_{p_1^{e_1}}\times \cdots \times K_{p_r^{e_r}}$ or $K_{2^e}\times K_{p_1^{e_1}}\times \cdots \times K_{p_r^{e_r}}$.

If $e \neq 0$, we define
\begin{equation}
\begin{array}{lcll}
\rho_0: & A_0 & \rightarrow & K_n \\
 & (a_0,a_1,\ldots,a_r) & \mapsto & k, 
\end{array} 
\label{eq18}
\end{equation}
where $k$ is solution of (\ref{eq15}), and if $e=0$, we define 
\begin{equation}
\begin{array}{lcll}
\rho_1: & A_1 & \rightarrow & K_n \\
 & (a_1,\ldots,a_r) & \mapsto & k, 
\end{array} 
\label{eq19}
\end{equation}
where $k$ is solution of (\ref{eq16}).

\begin{teo}
The functions $\rho_0$ and $\rho_1$ defined in (\ref{eq18}) and (\ref{eq19}), respectively, are bijective functions.
\label{Proposition36}
\end{teo}

\begin{proof}
Suppose first that $e\neq 0$. The proof for case $e=0$ is similar.

By Lemma \ref{lemma33}, $\rho_1$ is surjective.

Now let be $(a_1,\ldots,a_r),(b_1,\ldots,b_r) \in A$ such that $$\rho_1(a_1,\ldots,a_r)=\rho_1(b_1,\ldots,b_r),$$ then there exist $k$ such that $k \equiv a_i \pmod{l_i}$ and $k \equiv b_i \pmod{l_i}$ for all $i = 1,\ldots,r$. Consequently, $a_i \equiv b_i \pmod{l_i}$ for all $i=1,\ldots,r$; this is $(a_1,\ldots,a_r)=(b_1,\ldots,b_r)$ in $A$.

Therefore, $\rho_1$ is a bijective function.
\end{proof}

\begin{teo}
Let $\rho: A \rightarrow K_n$ be the function given by  
$$\rho = \left\{
\begin{array}{ll}
\rho_0  & \text{, if } e\neq 0,\\
\vspace{0.1cm} & \\
\rho_1  & \text{, if } e=0.
\end{array} \right.
$$
Thus, $\rho$ is a bijective function.
\label{theorem40}
\end{teo}

\begin{proof}
It is a consequence of Proposition \ref{Proposition36}.
\end{proof}

In \cite{rotman}, if $X$ be a set, $G$ be a group and $f:G \rightarrow X$ be a bijection, then there is a unique operation on $X$ so that $X$ is a group and $f$ is an isomorphism.

Therefore, $A$ is isomorphic to $K_n$; hence, knowing the cardinality of $K_n$ is enough to calculate the cardinality of $A$.

\section{Algorithms}
\label{sec:6}

In this Section we describe the algorithms we developed to be able to solve the congruence systems (\ref{eq22}) or (\ref{eq23}) and count all their solutions. In this way we were able to determine $K_n$ and $|K_n|$. 

We note that $1,-1 \in  K_n$, since
$$\rho_0(1,\ldots,1)=\rho_1(1,\ldots,1)=1,$$ and 
$$\rho_0(-1,\ldots,-1)=\rho_1(-1,\ldots,-1)=-1.$$

Therefore, $|K_n|\geq 2$ for all $n \in \mathbb{N}$.

\bigskip

Consider a system of two congruences, as in Proposition \ref{Proposition27}. If the congruence system has a solution, then, the following algorithm return those solution. For the case that congruence system has not solution, the algorithm return 0; $0$ was chosen as an output, since $0$ cannot be a solution of the congruence system (\ref{eq15}) or (\ref{eq16}).

\begin{center}
\begin{tabular}{l}
\hline 
\textbf{Algorithm 1.} Solution of the congruence system 
$\begin{array}{rl}
x \equiv & a \pmod p \\
x \equiv & b \pmod q .
\end{array}$
\\
\hline
\textbf{Input:} $a,b,p,q$.\\
\textbf{Output:} $x$.\\

\begin{tabular}{rl}
1. & $g \leftarrow$ greatest common divisor of $p$ and $q$. \\
2. & \textbf{if} $(a-b)\%g = 0$ \textbf{then} \\
3. & \hspace{0.25cm} With the extended Euclidean algorithm, we obtain \\
   & \hspace{0.25cm} $s$ and $t$ such that $g = ps+qt$. \\
4. & \hspace{0.25cm} $d\leftarrow (a-b)/g$ .\\
5. & \hspace{0.25cm} $x\leftarrow a+p(-sd)$ .\\
6. & \textbf{else} $x \leftarrow 0$. 
\end{tabular}
\end{tabular}
\end{center}

For Algorithm 2 and Algorithm 3, the arrays $ListKn$, $ListWn$, $S$ and $L$ are necessary, where$$ListKn = \left\{
\begin{array}{ll}
\left [ K_{2^e}, K_{p_1^{e_1}}, \ldots , K_{p_r^{e_r}} \right]  & \text{, if } e\neq 0,\\
\vspace{0.1cm} & \\
\left[ K_{p_1^{e_1}}, \ldots , K_{p_r^{e_r}} \right]  & \text{, if } e=0,
\end{array} \right.$$ and 
$$ListWn= \left\{
\begin{array}{ll}
\left[ l_0,l_1, \ldots, l_r \right]  & \text{, if } e\neq 0,\\
\vspace{0.1cm} & \\
\left[l_1, \ldots, l_r \right]  & \text{, if } e=0.
\end{array} \right.$$ 
$S$ and $L$ are arrays of length $r-1$ that are initialized with all their inputs $0$, and the value of each one of their entries changes according to the algorithm.

\begin{center}
\begin{tabular}[t]{l}
\hline 
\textbf{Algorithm 2. $Solution(y,l,i,S,L,K,r)$}.\\

Find the solution of the congruence system (\ref{eq15}) or (\ref{eq16}).
\\
\hline
\textbf{Input:} $y,l,i,S,L,K,r$.\\
\textbf{Output:} $K$.\\

\begin{tabular}{rl}
1. & \textbf{for} $k=0$ \textbf{to} $3$ \textbf{do} \\
2. & \hspace{0.25cm} \textbf{if} $ListKn[i+1][k]=0$ \textbf{then} \\
3. & \hspace{0.25cm} \hspace{0.25cm} \textbf{break} \\
   & \hspace{0.25cm} \textbf{else} \\
4. & \hspace{0.25cm} \hspace{0.25cm} $x1 \leftarrow $ with algorithm 1, solution of the congruence system \\
   & \hspace{0.25cm} \hspace{0.25cm} \hspace{0.25cm} $\begin{array}{rl}
x \equiv & y \pmod l \\
x \equiv & ListKn[i+1][k] \pmod{ListWn[i+1]} .
\end{array}$ 
 \\
5. & \hspace{0.25cm} \hspace{0.25cm} \textbf{if} $x1 \neq 0$ \textbf{then} \\
6. & \hspace{0.25cm} \hspace{0.25cm} \hspace{0.25cm} \textbf{if} $(i+1)=(r-1)$ \textbf{then} \\
7. & \hspace{0.25cm} \hspace{0.25cm} \hspace{0.25cm} \hspace{0.25cm} $K = K \cup \lbrace x1 \rbrace$. \\
8. & \hspace{0.25cm} \hspace{0.25cm} \hspace{0.25cm} \textbf{else} \\
9. & \hspace{0.25cm} \hspace{0.25cm} \hspace{0.25cm} \hspace{0.25cm} $S[i+1] \leftarrow x1$. \\
10. & \hspace{0.25cm} \hspace{0.25cm} \hspace{0.25cm} \hspace{0.25cm} $L[i+1] \leftarrow \mathrm{lcm}\, (l,ListWn[i+1])$.\\ 
11. & \hspace{0.25cm} \hspace{0.25cm} \hspace{0.25cm} \hspace{0.25cm} $K1 \leftarrow Solution(S[i+1],L[i+1],i+1,S,L,K,r)$. \\
12. & \hspace{0.25cm} \hspace{0.25cm} \hspace{0.25cm} \hspace{0.25cm} $K = K \cup K1$.
\end{tabular}
\end{tabular}
\end{center}

To solve the congruence system (\ref{eq15}), the first two are resolved, which are
$$\begin{array}{rl}
x \equiv & a_1 \pmod {l_1} \\
x \equiv & a_2 \pmod {l_2}.
\end{array}$$
In the case of having the solution, be $y \mod l$ such a solution, where $l$ is the least common multiple of $l_1,l_2$. In addition, $S=\left[a_1,y,0,\ldots,0 \right]$ and $L=\left[l_1,l,0,\ldots,0 \right] $; otherwise, different values are chosen for $a_1$, $a_2$. To continue solving the congruence system, consider 
$$\begin{array}{rl}
x \equiv & y \pmod {l} \\
x \equiv & a_3 \pmod {l_3},
\end{array}$$
and repeat the process until all the possible equations of form (\ref{eq15}) are considered.

In the same way, we solve the congruence system (\ref{eq16}).

Algorithm 2 recursively evaluates each of the possible combinations of vectors $(a_1,\ldots,a_r)$ in $K_{p_1^{e_1}}\times \cdots \times K_{p_r^{e_r}}$ or $(a_0,a_1,\ldots,a_r)$ in $K_{2^e}\times K_{p_1^{e_1}}\times \cdots \times K_{p_r^{e_r}}$ to find all the solutions of the congruence system (\ref{eq15}) or (\ref{eq16}).

\begin{center}
\begin{tabular}{l}
\hline 
\textbf{Algorithm 3.} Cardinality of $K_n$. \\
\hline
\textbf{Input:} $n \in \mathbb{N}$.\\
\textbf{Output:} $K_n$ and $|K_n|$.\\

\begin{tabular}{rl}
1. & Factorize $n$, $n \leftarrow p_1^{e_1} \cdots p_r^{e_r}$ such that $p_1 < \cdots < p_r$. \\
2. & $PR \leftarrow [p_1,\ldots,p_r]$, $PO\leftarrow [e_1,\ldots,e_r]$, $ListWn\leftarrow [0,\ldots,0]$,\\
& $ListKn \leftarrow [0, \ldots, 0]$, $K\leftarrow \lbrace \rbrace$, $S \leftarrow [0,\ldots,0]$ and $L \leftarrow [0,\ldots,0]$. \\
3. & \textbf{if} $PR[0]=2$ \textbf{then} \\
4. & \hspace{0.25cm} \textbf{if} $PO[0]<3$ \textbf{then} $ListWn[0]\leftarrow 3\cdot 2^{PO[0]-1}$. \\
5. & \hspace{0.25cm} \hspace{0.25cm} \textbf{else} $ListWn[0]\leftarrow 3\cdot 2^{PO[0]-2}$. \\
6. & \hspace{0.25cm} $ListKn[0]\leftarrow [1,-1,0,0]$. \\
7. & \hspace{0.25cm} \textbf{for} $i=1$ \textbf{to} $r-1$ \textbf{do} \\
8. & \hspace{0.25cm} \hspace{0.25cm} $ListWn[i]\leftarrow PR[i]^{PO[i]-1} \cdot \left( PR[i]^2-1\right)/2$. \\
9. & \hspace{0.25cm} \hspace{0.25cm} \textbf{if} $i=1$ and $PR[i]=3$ \textbf{then} $ListKn[i]\leftarrow [1,-1,0,0]$. \\
10. & \hspace{0.25cm} \hspace{0.25cm} \textbf{else} \textbf{if} $PO[i]=1$ \textbf{then} $ListKn[i]\leftarrow [1,-1,PR[i],-PR[i]]$. \\
11. & \hspace{0.25cm} \hspace{0.25cm} \hspace{0.25cm} \textbf{else} $ListKn[i]\leftarrow [1,-1,0,0]$. \\
12. & \textbf{if} $PR[0]>2$ \textbf{then} \\
13. & \hspace{0.25cm} \textbf{for} $i=0$ \textbf{to} $r-1$ \textbf{do} \\
14. & \hspace{0.25cm} \hspace{0.25cm} $ListWn[i]\leftarrow PR[i]^{PO[i]-1} \cdot \left( PR[i]^2-1\right)/2$. \\
15. & \hspace{0.25cm} \hspace{0.25cm} \textbf{if} $i=0$ and $PR[i]=3$ \textbf{then} $ListKn[i]\leftarrow [1,-1,0,0]$. \\
16. & \hspace{0.25cm} \hspace{0.25cm} \hspace{0.25cm} \textbf{else} \textbf{if} $PO[i]=1$ \textbf{then} $ListKn[i]\leftarrow [1,-1,PR[i],-PR[i]]$. \\
17. & \hspace{0.25cm} \hspace{0.25cm} \hspace{0.25cm} \hspace{0.25cm} \textbf{else} $ListKn[i]\leftarrow [1,-1,0,0]$. \\
18. & $Wn \leftarrow $ least common multiple of $ListWn[0],\ldots,ListWn[r-1]$. \\
19. & \textbf{for} $j=0$ \textbf{to} $3$ \textbf{do}\\
20. & \hspace{0.25cm} \textbf{if} $ListKn[0][j]\neq 0$ \textbf{then} \\
21. & \hspace{0.25cm} \hspace{0.25cm} $S[0] \leftarrow ListKn[0][j]$. \\
22. & \hspace{0.25cm} \hspace{0.25cm} $L[0] \leftarrow ListWn[0]$. \\
23. & \hspace{0.25cm} \hspace{0.25cm} $R \leftarrow Solution(S[0],L[0],0,S,L,K,r)$. \\
24. & \hspace{0.25cm} \hspace{0.25cm} $K \leftarrow K \cup R$.\\
25. & $K_n \leftarrow K$ and  $|K_n|\leftarrow$ cardinality of $K$.
\end{tabular}
\end{tabular}
\end{center}


From Theorem \ref{theorem19}, Proposition \ref{Proposition25}, Theorem \ref{theorem31}, the definition of $A$ in (\ref{eq17}), Proposition \ref{Proposition36}, Algorithm 1 and Algorithm 2, Algorithm 4 was designed to determine the cardinality of $K_n$.



Of Proposition \ref{Proposition32}, we have $|G_n|=|\mathbb{Z}_{w(n)}^{*}|/|\mathrm{Ker}\,(\psi)|$. Also $|\mathbb{Z}_{w(n)}^{*}| = \varphi(w(n))$, where $\varphi$ is Euler's phi function. Since $|\mathrm{Ker}\,(\psi)|=|K_n|$ as obtained from Algorithm 3, then the cardinality of $G_n$ can be calculated in the following way.

\begin{center}
\begin{tabular}{l}
\hline 
\textbf{Algorithm 4.} Cardinality of $G_n$. \\
\hline
\textbf{Input:} $n$.\\
\textbf{Output:} $|G_n|$.\\

\begin{tabular}{rl}
1. & In Algorithm 3, we obtain $Wn$ and $|K_n|$. \\
2. & $|G_n| \leftarrow \varphi(Wn)/|K_n|$,where $\varphi$ is the function $\varphi$ of Euler. 
\end{tabular}
\end{tabular}
\end{center}

\footnotesize

\vskip3mm \noindent Leticia Pe\~na T\'ellez:\\ Facultad de Ciencias, Universidad  Aut\'onoma del Estado de M\'exico\\ 
{\tt lpenat003@alumno.uaemex.mx}

\vskip3mm \noindent Martin Ort\'iz Morales:\\ Facultad de Ciencias, Universidad  Aut\'onoma del Estado de M\'exico\\
{\tt mortizmo@uaemex.mx}


\begin{thebibliography}{20}

\bibitem{atiyah}
Atiyah M. F. and Macdonald I. G.,Introduction to Commutative Algebra, Reading, Massachusetts - Menlo Park, California, Addison - Wesley Publishing Co., 1969.

\bibitem{dickson}
Dickson L.E. First course in the theory of equations. New York : John Wiley $ \& $ Sons, 1922.

\bibitem{nobauer}
Lausch, H., Muller W. B. and N\"obauer W. {\itshape \"Uber die Struktur einer durch Dicksonpolynome dragestellten Permutationsgruppe des Restklassenringes modulo $n$}. J. reine angew. Math. \textbf{261} (1973), 88 - 99.

\bibitem{dicksonpolynomials}
Lidl, R., Mullen, G. L. and Turnwald, G. Dickson polynomials. Longman Scientific and Technical ; New York : Copublished in the United States with John Wiley $ \& $ Sons, Harlow, Essex, England, 1993.

\bibitem{biniflam} 
Mc Donald B. R., Finite Rings with Identity, Inc. New York, Marcell Dekker, 1974.

\bibitem{niven}
Niven I.,Zuckerman H., Montgomery H. An Introduction to the Theory of Numbers. Jhon Wiley \& and Sons, Inc. 1991.

\bibitem{leticia}
Pe\~na T., L. (2014){\itshape Cifrado de datos e intercambio de claves utilizando polinomios de Dickson} (Tesis de Maestr\'ia). Universidad Aut\'onoma Metropolitana, M\'exico, Distrito Federal.

\bibitem{rotman}
Rotman J. J. An introduction to the theory of groups. Springer-Verlag 1995.

\bibitem{schur}
Schur, I.{\itshape \"Uber den Zusammenhang zwischen einem Problem der Zahlentheorie und einem Satz \"uber algebraische Funktionen, Sitzungsber}. Preuss. Akad. Wiss. Berlin 1923,123-134.

\end{thebibliography}
\end{document}